\documentclass[twoside,12pt]{article}
\usepackage{amsbsy,latexsym,amsopn,amstext,amsxtra,euscript,amscd}
\usepackage{amssymb,amsfonts,amsmath,amsthm}
\usepackage{graphics,graphicx}
\usepackage{enumitem}
\graphicspath{{Fig/},{Fig3D/}}

\newtheorem{theorem}{Theorem}
\newtheorem{corollary}{Corollary}
\newtheorem{identity}{Identity}[theorem]

\newcommand*{\ADRn}{University of Sopron,  Institute of Mathematics, Hungary. \texttt{nemeth.laszlo@uni-sopron.hu}}

\newcommand*{\TIT}{Tilings of $(2\times2\times n)$-board with colored cubes and bricks}
\title{\bf \TIT}

\author{L\'aszl\'o N\'emeth\footnote{\ADRn}}
\date{}

\begin{document}
	 
\maketitle \thispagestyle{empty}

\begin{abstract}
Several articles deal with tilings with squares and dominoes on 2-dimensional boards, but only a few on boards in 3-dimensional space. We examine a tiling problem with colored cubes and bricks of $(2\times2\times n)$-board in three dimensions. After a short introduction and the definition of breakability we show  a way to get the number of the tilings of an $n$-long board considering the $(n-1)$-long board. It describes recursively the number of possible breakable and unbreakable tilings. Finally, we give some identities for the recursions using breakability. The method of determining the recursions in space can be useful in mathematical education as well.  \\[1mm]
 {\em Key Words: Tiling in space, recurrence sequence, combinatorial identity.}\\
 {\em MSC code:  Primary 05B45; Secondary 05A19, 11B37, 11B39, 52C22.} \\
  The final publication is available at International Journal of Mathematical Education in Science and Technology via https://www.tandfonline.com/toc/tmes20/current.
\end{abstract}




\section{Introduction}

Let $r_n$ be the number of the different tilings of a board with $(1\times 1)$-squares and $(1\times 2)$-dominoes.
It is known that the number of the tilings of a $(1\times n)$-board on the square mosaic is given by the Fibonacci numbers \cite{BQ,BPS}. In fact, $r_n=F_{n+1}$, where $(F_n)_{n=0}^\infty$ is the  Fibonacci sequence ({A000045} in the OEIS \cite{oeis}).

McQuistan and Lichtman \cite{ML} (generalizations by Kahkeshani \cite{Kah}) studied the number $r_n$ of the tilings of a $(2\times n)$-board, and they proved that $r_n$ satisfies the identity
\begin{equation}\label{eq:til_44}
r_n=3r_{n-1}+r_{n-2}-r_{n-3}
\end{equation}
for $n\geq3$ with initial values $r_0=1$, $r_1=2$, and $r_2=7$ ({A030186} in the OEIS \cite{oeis}). 

Benjamin and Quinn \cite{Ben} gave the generalized Fibonacci sequence $(u_n)_{n=0}^\infty$, where
\begin{equation} \label{eq:u}
u_{n}=a u_{n-1}+b u_{n-2},\quad (n\geq 2)
\end{equation}
with initial values $u_0=1$, $u_1=a$,
so that $u_n$ is interpreted as the number of ways to tile a $(1 \times n)$-board using $a$ colors of squares and $b$ colors of dominoes. Obviously, if $a=b=1$, then $u_n=F_{n+1}$. Belbachir and Belkhir \cite{Belb} proved some general combinatorial identities related to $u_n$.

Let $R_n$ be the number of tilings of a $(2\times n)$-board using $a$ colors of squares and $b$ colors of dominoes.  Katz and Stenson \cite{KS} showed the recurrence rule  
\begin{equation}\label{eq:tilcol_44}
R_n=(a^2+2b)R_{n-1}+a^2b\,R_{n-2}-b^3R_{n-3},\quad (n\geq3)
\end{equation}
with  $R_0=1$, $R_1=a^2+b$, and $R_2=a^4+4a^2b+2b^2$. 

Komatsu et~al.\ \cite{til_hyp} generalized the tilings of a $(2\times n)$-board for all regular squared mosaics with Schl\"afli's symbol $\{4,q\}$ ($q\geq4$). In  the case $q>4$, the mosaic is realized in the hyperbolic plane, and if $q=4$, then it is in the Euclidean plane. They provided that 
the sequence $(R_n)_{n=0}^{\infty}$ satisfies  the fourth-order linear homogeneous recurrence relation 
\begin{equation}\label{eq:main_color}
R_n= \alpha_q\,R_{n-1} +\beta_q\,R_{n-2}+ \gamma_q\,R_{n-3} -b^{2(q-2)} R_{n-4},\quad(n\geq4)
\end{equation}
where 
\begin{eqnarray*}
	\alpha_{q+2} &=& a \alpha_{q+1}+b \alpha_{q},\label{eqs:recur_alfa}\\
	\beta_{q+3} &=& (a^2+b)\beta_{q+2}+b(a^2+b)\beta_{q+1}-b^3\beta_{q},\label{eqs:recur_beta}\\
	\gamma_{q+2} &=& -ab\gamma_{q+1}+b^3\gamma_{q}, \label{eqs:recur_gamma}
\end{eqnarray*}
with initial values 
\begin{eqnarray*}
	& &	\alpha_4=a^2+b,\  \alpha_5=a(a^2+3b),\\
	& &	\beta_4=2b(a^2+b),\  \beta_5=b(a^2+b)(a^2+2b),\  \beta_6=b(a^6+6a^4b+10a^2b^2+2b^3),\\
	& &	\gamma_4=b^2(a^2-b),\  \gamma_5=-ab^3(a^2+b),
\end{eqnarray*}
moreover
$R_0=1$,
$R_1=u_{q-2}$,
$R_2=u_{q-2}^2+abu_{q-4}u_{q-3}+bu_{q-3}^2+b^2u_{q-4}^2$,
$R_3=(u_{q-2}^2+2abu_{q-4}u_{q-3}+2bu_{q-3}^2+2b^2u_{q-4}^2)u_{q-2}
+b^2(u_{q-3}u_{q-4}+(a^2+b)u_{q-4}u_{q-5}
+au_{q-4}^2)u_{q-3}
+ab^3u_{q-4}^2u_{q-5}$.

If $q=4$, then they showed that  \eqref{eq:main_color} returns with \eqref{eq:tilcol_44}.

\begin{figure}[!h]
	\centering
	\includegraphics{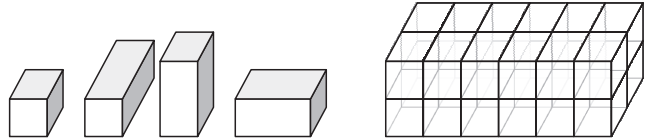} 
	\caption{Cube, bricks, and $(2\times 2\times 6)$-board}
	\label{fig:Cube_bricks}
\end{figure}

In this article, we examine the tilings of a $(2\times 2\times n)$-board with $(2\times 2\times 1)$-layers labeled $1, 2, \ldots, n$  on Euclidean cube mosaic using colored $(1\times 1\times 1)$-cubes and $(2\times 1\times 1)$-, $(1\times 2\times 1)$- or $(1\times 1\times 2)$-bricks (union of two adjacent cubes) as a generalization of the tiling above for  3-dimensional space (Figure~\ref{fig:Cube_bricks}). A layer consists of four cells.
From this point on, let $R_n$ be the number of the different tilings
with $a$ colors of cubes and $b$ colors of bricks  of a $(2\times 2\times n)$-board. If $a=1$ and $b=1$, then let $R_n$ be denoted by $r_n$.

We obtain the following main theorem.
\begin{theorem}[Main Theorem]\label{th:maincolor3d}
The sequence $(R_n)_{n=0}^{\infty}$ satisfies  the sixth-order linear homogeneous recurrence relation 
	\begin{equation}\label{eq:main_color3d}
	R_n= \sum_{i=1}^{6}\alpha_i\,R_{n-i} \quad(n\geq6),
	\end{equation}
where 
	\begin{eqnarray*}
	\alpha_1 &=& a^4+7a^2b+6b^2,\\
	\alpha_2 &=& b(a^6+6a^4b+6a^2b^2-7b^3),\\
	\alpha_3 &=& -2b^3(a^6+5a^4b+13a^2b^2+4b^3), \\
	\alpha_4 &=& b^5(a^6+2a^4b+6a^2b^2+9b^3),\\
	\alpha_5 &=& b^8(a^4-a^2b+2b^2),\\
	\alpha_6 &=& -b^{12},
	\end{eqnarray*}
\noindent and the initial values $R_n$ ($n=0,\dots, 5$) are 
$1$, 
$a^4+4a^2b+2b^2$, 
$a^8+12a^6b+42a^4b^2+44a^2b^3+9b^4$,
$a^{12}+20a^{10}b+142a^8b^2+440a^6b^3+588a^4b^4+288a^2b^5+32b^6$,
$a^{16}+28a^{14}b+306a^{12}b^2+1672a^{10}b^3+4863a^8b^4+7416a^6b^5+5470a^4b^6
+1620a^2b^7+121b^8$, and
$a^{20}+36a^{18}b+534a^{16}b^2+4248a^{14}b^3+19774a^{12}b^4+55200a^{10}b^5
+91200a^8b^6+84984a^6b^7+40553a^4b^8+8204a^2b^9+450b^{10}$ (respectively).
\end{theorem}

If we use only one color for the tilings (i.e., example gray as in our figures), so $a=1$ and $b=1$, then we have the main corollary.   
\begin{corollary}\label{cor:maincolor3d}
 The sequence $(r_n)_{n=0}^{\infty}$ satisfies  the sixth-order linear homogeneous recurrence relation 
	\begin{equation}\label{eq:main_color3d_ab1}
	r_n= 14 r_{n-1} +6 r_{n-2}-46 r_{n-3} + 18r_{n-4} + 2 r_{n-5} - r_{n-6},\quad(n\geq6)
	\end{equation}
where the initial values $r_i$ ($n=0,\dots, 5$) are $1$, $7$, $108$, $1511$, $21497$, $305184$, and $4334009$ ({A033516}).
\end{corollary}

Let us consider the tilings with bricks exclusively. 
Now, $a=0$.  Obviously, the relation \eqref{eq:main_color3d} holds for the sequence of the number of the tilings, but we can formulate another theorem as well.
\begin{theorem}\label{th:bricks3d}
	The sequence $(R_n)_{n=0}^{\infty}$ satisfies  the third-order linear homogeneous recurrence relation 
	\begin{equation}\label{eq:brics3d}
	R_n= 3b^2\, R_{n-1} +3b^4\, R_{n-2}-b^6\, R_{n-3} \quad(n\geq3),
	\end{equation}
	where initial values $R_0=1$, $R_1=2b^2$, and  $R_2=9b^4$.
\end{theorem}
If $a=0$ and $b=1$, then the sequence $(R_n)=(1, 2,  9,  32, 121, 450, 1681, 6272, 23409, \ldots )$ appears as  {A006253} in the OEIS \cite{oeis}. 

\section{Tilings of  $(2\times 2\times n)$-board}

For ease of notation, let us denote the $(2\times 2\times n)$-board by $\mathcal{B}_n$. Thus, a layer is denoted by $\mathcal{B}_1$.
Now, we define the \textit{breakability} of a tiling of $\mathcal{B}_n$. 
A tiling is \textit{breakable} at layer~$i$ if it can be split up into two subtilings, one covering layers 1 through $i$ and the other covering layers ($i+1$) through $n$.  In other words,  this tiling is a concatenation of two possible tilings of the subboard $\mathcal{B}_i$ and the subboard $\mathcal{B}_{n-i}$. Clearly, the number of colored tilings of such a board is $R_iR_{n-i}$. 
Otherwise, the tiling is \textit{unbreakable} at layer~$i$. A tiling is unbreakable at layer~$i$ (or at position~$i$), if and only if at least one $(1\times 1\times 2)$-brick has the part in layer~$i$ and layer~$(i+1)$ as well. In other words, a brick covers the position~$i$ or a brick overhangs the layer~$i$ (see Figure~\ref{fig:unbreakable3d_n2}).
If a tiling is unbreakable at all the positions, then we say that it is unbreakable. We consider all the tilings unbreakable at positions~$0$ and $n$. 
Let $\widetilde{R}_n$ be the number of the different unbreakable tilings with $a$ colors of cubes and $b$ colors of bricks of $\mathcal{B}_n$. If $a=1$ and $b=1$, then let $\widetilde{R}_n$ be denoted by $\widetilde{r}_n$.

\subsection{Proof of theorems}

\subsubsection{Proof of Theorem~\ref{th:maincolor3d}}

Figure~\ref{fig:tiling3d_n1} illustrates all the seven different tilings of $\mathcal{B}_1$ with cubes and bricks without coloring. If we use $a$ colors of cubes and $b$ colors of bricks, then according to the tilings from left to right in Figure~\ref{fig:tiling3d_n1}, we have $b^2$, $b^2$, $a^2b$, $a^2b$, $a^2b$, $a^2b$, and  $a^4$ different tilings. Since they are all unbreakable, $R_1=\widetilde{R}_1=a^4+4a^2b+2b^2$ and $r_1=\widetilde{r}_1=7$. Moreover, let $R_0=\widetilde{R}_0=r_0=\widetilde{r}_0=1$, because we can tile  the empty board only one way, which is unbreakable as well.   

\begin{figure}[!h]
	\centering
	\includegraphics{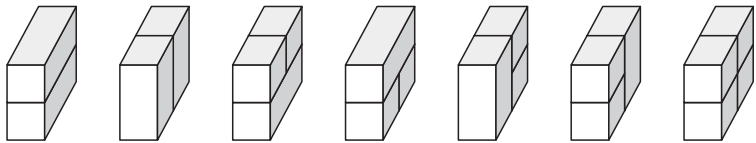} 
	\caption{Tilings of $\mathcal{B}_1$}
	\label{fig:tiling3d_n1}
\end{figure}

Considering $\mathcal{B}_2$, there are breakable tilings. Obviously, we obtain them by the concatenation of two tilings of $\mathcal{B}_1$  (see Figure~\ref{fig:breakable3d_n2}). Thus, their number is $7\cdot7$, but with colored cubes and bricks their number grows to $(a^4+4a^2b+2b^2)^2$.
\begin{figure}[!h]
	\centering
	\includegraphics{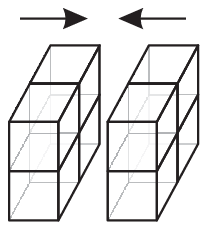} 
	\caption{Breakable tilings of $\mathcal{B}_2$}
	\label{fig:breakable3d_n2}
\end{figure}

For its unbreakable tilings we differentiate five cases. 
The first case is when exactly one brick covers position~1. (The leftmost  sub-figure in Figure~\ref{fig:unbreakable3d_n2} shows this case.) The other cells (white cubes in the figure) of the board can be cubes and bricks, and in both layers, they give three different subtilings --- one with three cubes and two tilings with one cube and one brick. As the covering brick can have four places we can see that the number of tilings in this case is $4\cdot3\cdot3$ and $4b(a^3+2ab)^2$, respectively, without and with coloring. 
The second case is when there are exactly two bricks with a common face overhanging in layer~$1$ (second sub-figure in Figure~\ref{fig:unbreakable3d_n2}). Now, there are $4\cdot2\cdot2$ such a tiling without coloring and  $4b^2(a^2+b)^2$ with coloring. 
In the third case, when the two overhanging bricks have a common edge, the number of tilings is $2b^2a^4$ (and 2 without coloring). 
The fourth case contains three covering bricks $4b^3a^2$ ways. 
And finally, when we tile with the maximum four bricks, then the number of possible tilings is $b^4$.

\begin{figure}[!h]
	\centering
	\includegraphics{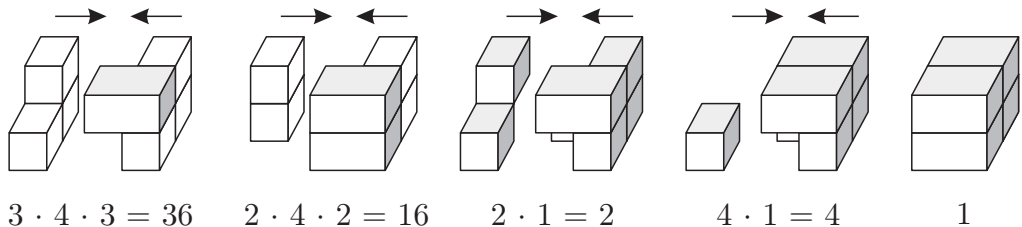} 
	\caption{Unbreakable tilings of $\mathcal{B}_2$}
	\label{fig:unbreakable3d_n2}
\end{figure}

Summarizing the results, we have $R_2=a^8+12a^6b+42a^4b^2+44a^2b^3+9b^4$, $r_2=108$, and $\widetilde{R}_2=4a^6b+22a^4b^2+28a^2b^3+5b^4$, $\widetilde{r}_2=59$.

Now, we define some new different types of subboards based on the six different types of tilings discussed earlier considering $\mathcal{B}_2$. 

\begin{figure}[!h]
	\centering
	\includegraphics{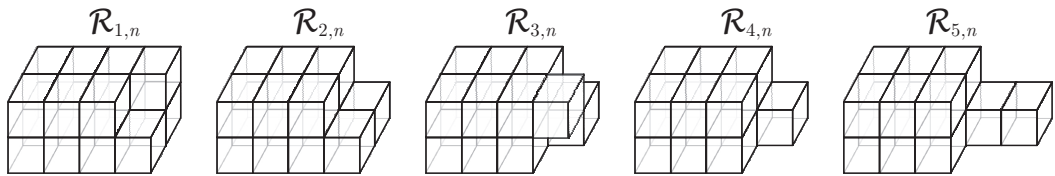} 
	\caption{Boards $\mathcal{R}_{j,n}$ when $n=4$ and $j=1,\ldots,5$}
	\label{fig:tiling3d_Rn}
\end{figure}

We let $\mathcal{R}_{1,n}$ denote the boards, when one cell is omitted from the $n$th layer of $\mathcal{B}_{n}$ (see Figure~\ref{fig:tiling3d_Rn}). The number of $\mathcal{R}_{1,n}$ is four because one cell can be deleted four different ways from the last layer  (see Figure~\ref{fig:tiling3d_Rn}).
Similarly, let $\mathcal{R}_{2,n}$, $\mathcal{R}_{3,n}$, and $\mathcal{R}_{4,n}$ denote the boards, respectively, when we delete from the $n$th layer of $\mathcal{B}_{n}$ two cells with a common face, two cells with only a common edge and three cells (Figure~\ref{fig:tiling3d_Rn}). The number of $\mathcal{R}_{2,n}$, $\mathcal{R}_{3,n}$,  and $\mathcal{R}_{4,n}$ is 4, 2, and 4, respectively.
Let $\mathcal{R}_{5,n}$ be the board obtained by joining a cell of layer~$(n+1)$ to an board $\mathcal{R}_{4,n}$, so that they have a common face. (The last two cells form a brick which covers the position~$n$, as the last sub-figure  of Figure~\ref{fig:tiling3d_Rn} shows.) The number of such boards is four. We shall always tile the last two cells of $\mathcal{R}_{5,n}$ with a brick according to the last sub-figure of Figure~\ref{fig:unbreakable3d_n2}, otherwise all the tilings of $\mathcal{R}_{4,n}$ could be the subtilings of the tiling of $\mathcal{R}_{5,n}$. Finally, $\mathcal{B}_{n}$ is denoted by $\mathcal{R}_{0,n}$.
Obviously, a tiling of all the new boards can be breakable or unbreakable at position~$(n-1)$ except in the case of $\mathcal{R}_{5,n}$. It is breakable, when $n\geq2$.

Moreover, we let $R_{0,n}$, $R_{1,n}$, $R_{2,n}$, $R_{3,n}$, $R_{4,n}$, and $R_{5,n}$ denote the number of tilings of board $\mathcal{R}_{0,n}$, $\mathcal{R}_{1,n}$, $\mathcal{R}_{2,n}$, $\mathcal{R}_{3,n}$, $\mathcal{R}_{4,n}$, and $\mathcal{R}_{5,n}$, respectively. The values of $\widetilde{R}_{0,n}$, $\widetilde{R}_{1,n}$, $\widetilde{R}_{2,n}$, $\widetilde{R}_{3,n}$, $\widetilde{R}_{4,n}$, and $\widetilde{R}_{5,n}$ are the numbers of the appropriate unbreakable tilings.

Examine  Figure~\ref{fig:tiling3d_n1} and \ref{fig:unbreakable3d_n2}   again. One can easily see  that if $n=1$, then we have 
$R_{0,1}=\widetilde{R}_{0,1}=R_1=\widetilde{R}_1=a^4+4a^2b+2b^2$, 
$R_{1,1}=\widetilde{R}_{1,1}=4(a^3+2ab)$,
$R_{2,1}=\widetilde{R}_{2,1}=4(a^2+b)$,
$R_{3,1}=\widetilde{R}_{3,1}=2a^2$,
$R_{4,1}=\widetilde{R}_{4,1}=4a$,
$R_{5,1}=\widetilde{R}_{5,1}=4b$. Furthermore, $\widetilde{R}_{5,n}=0$ for $n\geq2$.

Now, we give $R_{j,n}$ ($j=0,\ldots,5$) recursively for $n\geq2$. 
Figure~\ref{fig:table_recursio3d} illustrates how $\mathcal{R}_{j,n}$ can be built from $\mathcal{R}_{k,n-1}$ ($j,k=0,\ldots,5$, $n\geq2$). Row~0 contains the tiled boards in the case $n-1$, and column~0 does the tiled boards in the case $n$ while the other items of the table give the connection between them. 

\begin{figure}[!h]
	\centering
	\includegraphics{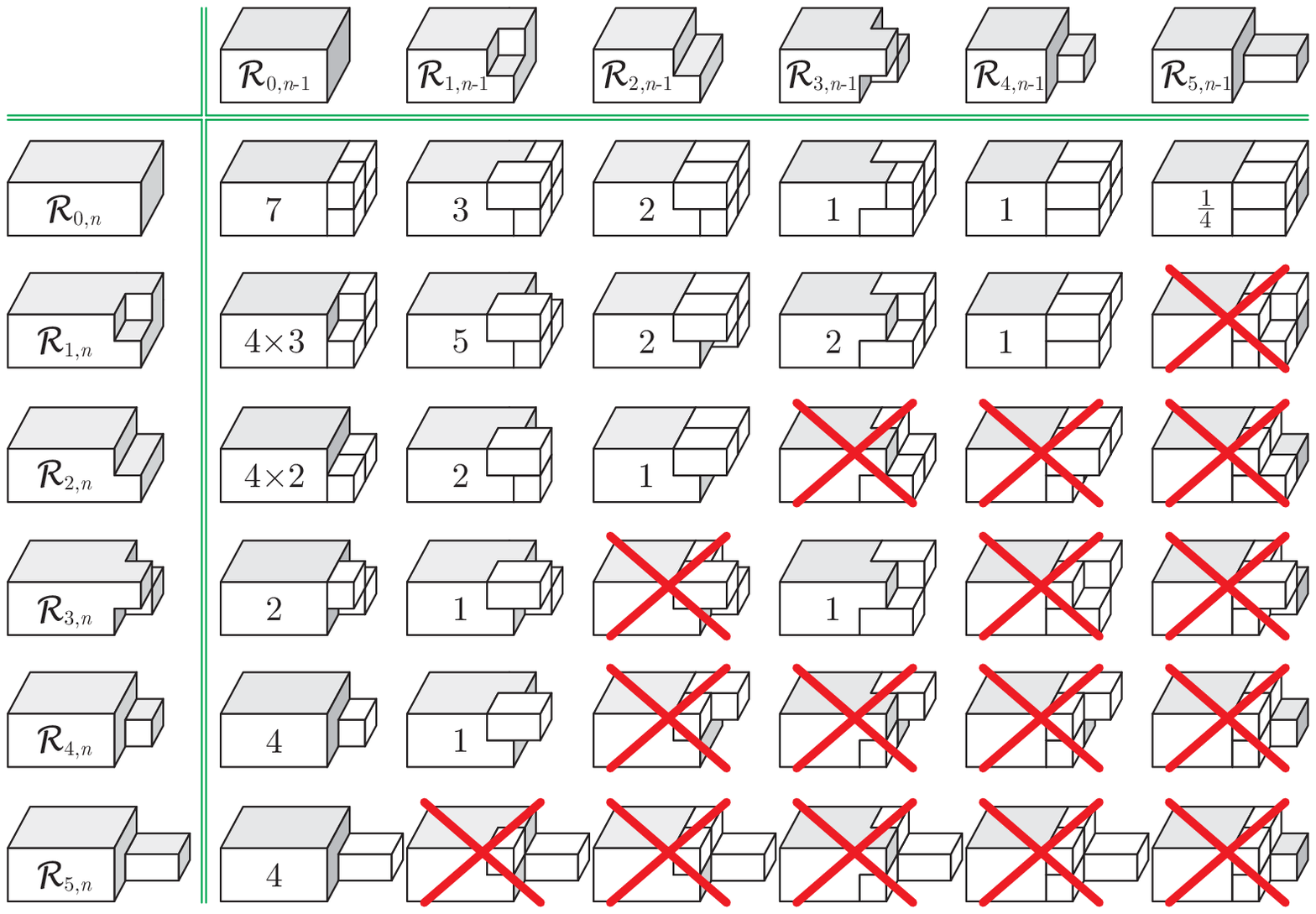} 
	\caption{Table for recursion without coloring}
	\label{fig:table_recursio3d}
\end{figure}

Let us examine the first row of this table. The first row shows the possible structures of $\mathcal{R}_{0,n}$ considering the last layers. (Of course, in the case $n=2$, it gives back the tilings discussed above.)  

\renewcommand{\labelitemi}{$-$}
\begin{itemize}[noitemsep,topsep=0pt]
\item The first item illustrates that the tilings of $\mathcal{R}_{0,n}$ can be built from the tilings of $\mathcal{R}_{0,n-1}$ seven different ways (see Figure~\ref{fig:tiling3d_n1}). Taking into consideration the coloring, they are $a^4+4a^2b+2b^2$ altogether.

\item  We take a tiling of $\mathcal{R}_{1,n-1}$ and complete it to a tiling of $\mathcal{R}_{0,n}$. If we put a cube into the missing cell's place in the $(n-1)$th layer, then we would get the first case (considering $\mathcal{R}_{0,n-1}$), and we would not get new types of tilings. So, we must put here an overhanging brick. The other cells can be covering cubes and bricks 3 different ways, with coloring $b(a^3+2ab)$ ways.

\item If we put two cubes or only one cube into the $(n-1)$th layer of $\mathcal{R}_{2,n-1}$, then we would get back the previous two cases, respectively. Thus, for a new type of tilings, we must cover the position~$n-1$ with two overhanging bricks. The other two cells must be cubes or a brick. So the number of tilings with coloring, building from $\mathcal{R}_{2,n-1}$ is $b^2(a^2+b)$. 

\item Similarly, in the case of $\mathcal{R}_{3,n-1}$, we need two overhanging bricks, and the number of tilings  is $a^2b^2$. 

\item Analogously, $\mathcal{R}_{4,n-1}$ provides $ab^3$ new tilings.

\item  Finally, considering the last item in row~1, with 3 additional bricks we get a tiling of $\mathcal{R}_{0,n}$ from $\mathcal{R}_{5,n-1}$. For the reason that all the four $\mathcal{R}_{5,n-1}$ will generate the same tiling, we have to divide the number of tilings of such a board by 4 because of the multiplicity. With coloring, it is $\frac14b^3$. All the other tilings originating from $\mathcal{R}_{5,n-1}$, when we use at least one cube, have been realized among the previous cases.  
\end{itemize}

Studying the other items of Figure~\ref{fig:table_recursio3d} we gain more connections between the tilings of the boards. The items crossed with red lines do not provide new types of tilings. They are partly in another type or the connection is not realizable. 

Summarizing the results, we have the system of homogeneous recurrence equations
\begin{equation}\label{eq:rec_system3d}
R_{j,n}=\sum_{k=0}^{5} m_{j,k}R_{k,n-1}, \qquad j=0, \ldots, 5,
\end{equation} 
where the matrix of the coefficients $m_{i,j}$ is 
$$\mathbf{M}=\begin{pmatrix}
a^4+4a^2b+2b^2 &b(a^3+2ab) &b^2(a^2+b)&a^2b^2& ab^3& \frac14 b^3\\
4(a^3+2ab)     & b(3a^2+2b)& 2ab^2    &2ab^2 & b^3 & 0\\
4(a^2+b)       & 2ab       & b^2      & 0    & 0   & 0\\
2a^2           & ab        & 0        & 0    & 0   & 0\\
4a             & b         & 0        & 0    & 0   & 0\\
4b             & 0         & 0        & 0    & 0   & 0\\
\end{pmatrix}. $$

As usual, the characteristic equation  
\begin{multline*}
x^6 + (-a^4-7a^2b-6b^2)x^5 + (-a^6b-6a^4b^2-6a^2b^3+7b^4)x^4 +\\ (2a^6b^3+10a^4b^4+26a^2b^5+8b^6)x^3 + (-a^6b^5-2a^4b^6-6a^2b^7-9b^8)x^2 +\\ (-a^4b^8+a^2b^9-2b^{10})x +b^{12} = 0
\end{multline*}
of $\mathbf{M}$ provides the recurrence relation for $(R_{j,n})$; see the proof of N\'emeth and  Szalay \cite[Lem.\ 2.1]{NSz_Power}. Recall $(R_{n})=(R_{0,n})$. (The computation was made by the help of software \textsc{Maple}.) Thus, we have relation \eqref{eq:main_color3d} of Main Theorem.

Moreover, we obtain the initial values of the recurrence for $n=2,3,4,5$ from the system \eqref{eq:rec_system3d}. Finally, we can check that relation \eqref{eq:main_color3d}  holds not only for the cases $n>6$ but also for $n=6$.  

\subsubsection{Proof of Theorem~\ref{th:bricks3d}}

Since the subboards $\mathcal{R}_{1,n}$, $\mathcal{R}_{3,n}$, and $\mathcal{R}_{4,n}$ can not be realized  with bricks exclusively, we delete the appropriate rows and columns from Figure~\ref{fig:table_recursio3d} and for tilings with bricks we gain Figure~\ref{fig:table_recursio_bricks3d}.

As with the previous proof, we obtain the matrix 
$$\mathbf{M}=\begin{pmatrix}
2b^2 & b^3 & \frac14 b^3\\
4b   & b^2 & 0 \\
4b   & 0   & 0
\end{pmatrix} $$
and its characteristic equation provides the recurrence relation \eqref{eq:brics3d}. 

\begin{figure}[!h]
	\centering
	\includegraphics{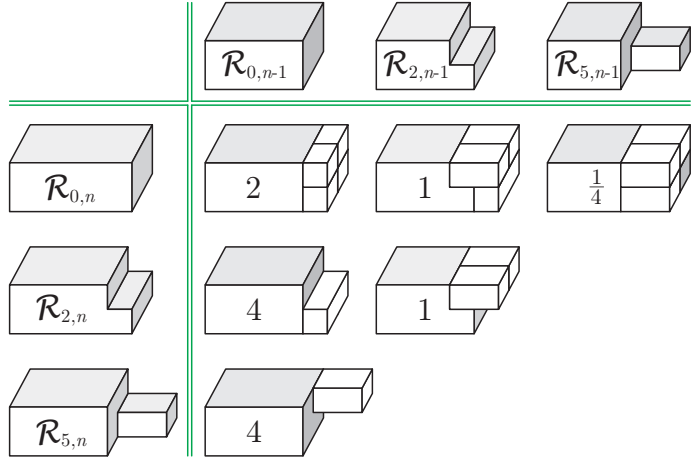} 
	\caption{Table for recursion of tilings with bricks}
	\label{fig:table_recursio_bricks3d}
\end{figure}

\subsection{Unbreakable tilings}

In this subsection, we determine the number of unbreakable tilings and using the results of the previous subsection we shall prove the following theorem and its corollary in the case $a=b=1$.

\begin{theorem}\label{th:unbreakable3d}
	The sequence $(\widetilde{R}_n)_{n=0}^{\infty}$ satisfies  the fourth-order linear homogeneous recurrence relation for $n\geq7$
\begin{equation}\label{eq:main_color3dun}
	\widetilde{R}_n= b(3a^2+4b)\,\widetilde{R}_{n-1} - 4b^4\,\widetilde{R}_{n-2} - 3a^2b^5\,\widetilde{R}_{n-3} + b^8\,\widetilde{R}_{n-4},
\end{equation}
where 
$\widetilde{R}_{0}=1$,
$\widetilde{R}_{1}=a^4+4a^2b+2b^2$,
$\widetilde{R}_{2}= 4a^6b+22a^4b^2+28a^2b^3+5b^4$ 
and the initial values of \eqref{eq:main_color3dun} are 
$\widetilde{R}_{3}= 12a^8b^2+80a^6b^3+158a^4b^4+88a^2b^5+4b^6$,  $\widetilde{R}_{4}= 36a^{10}b^3+288a^8b^4+776a^6b^5+798a^4b^6+252a^2b^7+4b^8$, 
$\widetilde{R}_{5}= 108a^{12}b^4+1008a^{10}b^5+3420a^8b^6+5112a^6b^7+3234a^4b^8+656a^2b^9+4b^{10}$, and
$\widetilde{R}_{6}= 324a^{14}b^5+3456a^{12}b^6+14112a^{10}b^7+27624a^8b^8+26576a^6b^9+11470a^4b^{10}+1644a^2b^{11}+4b^{12}$.
\end{theorem}

\begin{corollary}\label{cor:unbreakable3d}
	The sequence $(\widetilde{r}_n)_{n=0}^{\infty}$ satisfies  the fourth-order linear homogeneous recurrence relation for $n\geq7$
\begin{equation}\label{eq:main_color3dun11}
	\widetilde{r}_n= 7\,\widetilde{r}_{n-1} - 4\,\widetilde{r}_{n-2} - 3\,\widetilde{r}_{n-3} + \widetilde{r}_{n-4},
\end{equation}
	where 
	$\widetilde{r}_0=1$,
	$\widetilde{r}_1=7$, 
	$\widetilde{r}_2=59$, 
	$\widetilde{r}_3=342$,  
	$\widetilde{r}_4=2154$, 
	$\widetilde{r}_5=13542$, and 
	$\widetilde{r}_6=85210$.
	(This sequence has  not been in the OEIS, yet.)
\end{corollary}

\begin{proof}[Proof of Theorem~\ref{th:unbreakable3d}]
We recognized during the discussion of Figure~\ref{fig:table_recursio3d} that the tilings from $(R_{0,n-1})$ ($n\geq2$) and $(R_{5,n-1})$ are breakable if $n\geq3$. (We mention that if $n=2$, then there are $b^4$ unbreakable tilings with four overhanging bricks, which we will consider later  in our calculation.) So we have to delete the columns containing $(R_{0,n-1})$ and $(R_{5,n-1})$. In the other cases, we obtain unbreakable tilings, when $(R_{j,n-1})$ ($j=1,\ldots,4$) are unbreakable. Then the coefficient matrix of system \eqref{eq:rec_system3d} --- having omitted the sequence $(R_{5,n})$ --- in case of unbreakable tilings ($n\geq2$) is
$$\mathbf{U}=\begin{pmatrix}
0 &b(a^3+2ab) &b^2(a^2+b)&a^2b^2& ab^3\\
0 & b(3a^2+2b)& 2ab^2    &2ab^2 & b^3 \\
0 & 2ab       & b^2      & 0    & 0   \\
0 & ab        & 0        & 0    & 0   \\
0 & b         & 0        & 0    & 0   \\
\end{pmatrix}. $$

The values $\widetilde{R}_{j,1}$ ($j=0,1,\ldots,4$)  were also taken into consideration. Thus, 
$\widetilde{R}_{0,1}=\widetilde{R}_1$, 
$\widetilde{R}_{1,1}=4(a^3+2ab)$,
$\widetilde{R}_{2,1}=4(a^2+b)$,
$\widetilde{R}_{3,1}=2a^2$, and
$\widetilde{R}_{4,1}=4a$.

The characteristic equation of $\mathbf{U}$ is 
\begin{equation}
x(x^4-(3a^2b+4b^2)x^3+4b^4x^2+ 3a^2b^5x-b^8)=0,
\end{equation}
which provides relation \eqref{eq:main_color3dun}. Recall $(\widetilde{R}_{n})=(\widetilde{R}_{0,n})$.  The initial values come from the system~\eqref{eq:rec_system3d} for $n=2,3,4,5$, when $j$ and $k$ goes from $0$ to $4$. 
Do not forget that we do not get the tilings with four overhanging bricks of $\mathcal{B}_2$ (last sub-figure of Figure~\ref{fig:unbreakable3d_n2}). That is why finally we add $b^4$ to $\widetilde{R}_{2}$ and relation~\eqref{eq:main_color3dun} holds for $n\geq7$ only (and does not hold for $n=6$).  
\end{proof}

If we tile with bricks exclusively ($a=0$), then we have $\widetilde{R}_{0}=1$, $\widetilde{R}_{1}=2b^2$, $\widetilde{R}_{2}=5 b^4$, and $\widetilde{R}_{n}=4b^{2n}$, when $n\geq3$ (see Figure~\ref{fig:unbreakable3d_bricks}). 

\begin{figure}[!h]
	\centering
	\includegraphics[scale=0.9]{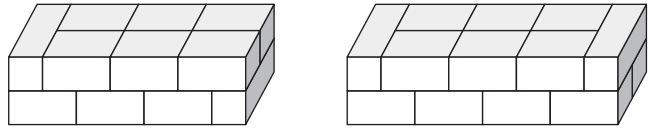}
	\caption{Unbreakable tilings with bricks when $n=7$ and $n=8$}
	\label{fig:unbreakable3d_bricks}
\end{figure}

\section{Some identities}

In the sequel, we give certain identities related to the sequences $(R_n)$ and $(\widetilde{R}_{n})$. The proofs are based on the tilings, not on recursive formulas.

\begin{identity} \label{id:Rn}
	If $n\geq1$, then
   \begin{equation*}	
      R_n=\sum_{i=0}^{n-1}R_i\widetilde{R}_{n-i}.	
   \end{equation*}
\end{identity}
\begin{proof}
As illustrated in Figure \ref{fig:identity_Rn}, let us consider the breakable colored tilings at layer~$i$ ($0\leq i <n$) of board $\mathcal{B}_n$, where the tilings on the right subboard $\mathcal{B}_{n-i}$ are unbreakable. The number of these tilings is $R_i\widetilde{R}_{n-i}$. If $i=0$, then the tilings are unbreakable on the whole $\mathcal{B}_n$. Clearly, when $i$ goes from 1 to $n-1$, we have different tilings and we consider all of them.  
\end{proof}
\begin{figure}[!h]
	\centering
	\includegraphics{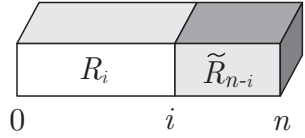}
	\caption{Tilings in case of Identity \ref{id:Rn}}
	\label{fig:identity_Rn}
\end{figure}

Now, we give equivalent formulas for $R_n$.  
\begin{identity} \label{id:Rn2}
  If $n\geq1$, then
	\begin{eqnarray*}	
	R_n&=&\sum_{i=1}^{n}R_{n-i}\widetilde{R}_{i},\\
	R_n&=&\frac12\sum_{i=0}^{n}R_{i}\widetilde{R}_{n-i},\\
	R_n&=&\frac12\sum_{i=0}^{n}R_{n-i}\widetilde{R}_{i}.	
	\end{eqnarray*}
\end{identity}

The next statement gives another rule of summation.
\begin{identity} \label{id:Rnm}
  If $m\geq1$ and $n\geq1$, then
  \begin{equation*}	
  R_{n+m}=R_{n}R_{m}+\sum_{i=1}^{n}\sum_{j=1}^{m}R_{n-i}R_{m-j}\widetilde{R}_{i+j}.	
  \end{equation*}
\end{identity}
\begin{proof}
Let us consider a $\mathcal{B}_{n+m}$ as the concatenation of $\mathcal{B}_{n}$ and $\mathcal{B}_{m}$.	First, we take the breakable tilings at layer~$n$, their cardinality is $R_nR_m$. Then we examine the unbreakable tilings at this layer. We cover the position~$n$ by $(i+j)$-long unbreakable tilings from position~$n-i$ to $n+j$ of $\mathcal{B}_{n+m}$. They give the remaining tilings. Figure~\ref{fig:identity_Rnm} illustrates these two cases.
\end{proof}
\begin{figure}[!h]
	\centering
	\includegraphics{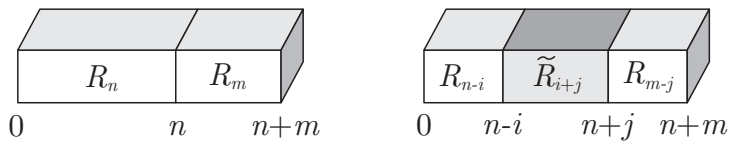}
	\caption{Tilings in case of Identity \ref{id:Rnm}}
	\label{fig:identity_Rnm}
\end{figure}

Identity \ref{id:Rnm} admits the following two remarkable specific cases by the choice of $m=1$ and $m=(k-1)n$, respectively.
\begin{identity} \label{id:Rnm1}
	If  $n\geq1$, then
	\begin{equation*}	
	R_{n+1}=R_{n}R_{1}+\sum_{i=1}^{n}R_{n-i}\widetilde{R}_{i+1}.	
	\end{equation*}
\end{identity}

\begin{identity} \label{id:Rnm2}
	If $n\geq1$  and $k\geq2$, then
	\begin{equation*}	
      R_{kn}=R_{n}R_{(k-1)n}+\sum_{i=1}^{n}\sum_{j=1}^{(k-1)n}R_{n-i}R_{(k-1)n-j}\widetilde{R}_{i+j}.	
	\end{equation*}
\end{identity}

In addition, from Identity \ref{id:Rnm} for a given $k$ ($0\leq k< n$) when we replace $n$ and $m$ by $n-k$ and $n+k$, respectively,  we obtain the following identity.

\begin{identity} \label{id:Rnm3}
	If  $0\leq k<n$, then
	\begin{equation*}	
	R_{2n}=R_{n-k}R_{n+k}+\sum_{i=1}^{n-k}\sum_{j=1}^{n+k}R_{n-k-i}R_{n+k-j}\widetilde{R}_{i+j}.	
	\end{equation*}
\end{identity}

Finally, we give an identity about the sum of the first $n$ terms of the sequence $(R_n)$.

\begin{identity} \label{id:Rsum3d}
  If $n\geq1$, then
	\begin{equation*}	
	\sum_{i=1}^{n}R_{i}= 
	   \sum_{i=1}^{n}\widetilde{R}_{i} \cdot   \sum_{j=0}^{n-i}{R}_{j}.	
	\end{equation*}	
\end{identity}
\begin{proof}
Let us fix $i$ ($1\leq i\leq n$) and consider all the possible boards of which the colored tilings are breakable at layer~$i$, and the tilings of the right subboard $\mathcal{B}_{i}$ are unbreakable as illustrated in Figure \ref{fig:identity_Rsum3d}. Then the sum of such tilings is $\widetilde{R}_{i}(R_0+R_1+\ldots + R_{n-i-1}+R_{n-i})$.  
Certainly, when $i$ goes from 1 to $n$, we have all the different tilings of boards $\mathcal{B}_{k}$ ($1\leq k \leq n$). Recall $\widetilde{R}_{0}=1$. 
\end{proof}

\begin{proof}
	Let us fix $i$ ($1\leq i\leq n$). Consider all the colored tilings of $\mathcal{B}_{j}$ $(i\leq j\leq n)$ which are breakable at position~$(j-i)$ and the tilings of the right subboard $\mathcal{B}_{i}$ (from position~$j-i$ to $j$) are unbreakable as illustrated in Figure \ref{fig:identity_Rsum3d}. 
	Then the sum of such tilings is $\widetilde{R}_{i}(R_0+R_1+\ldots + R_{n-i-1}+R_{n-i})$.  
	Certainly, when $i$ goes from 1 to $n$, we have all the different tilings of boards $\mathcal{B}_{i}$ ($1\leq i \leq n$). Recall $\widetilde{R}_{0}=1$. 
\end{proof}

\begin{figure}[!h]
	\centering
	\includegraphics{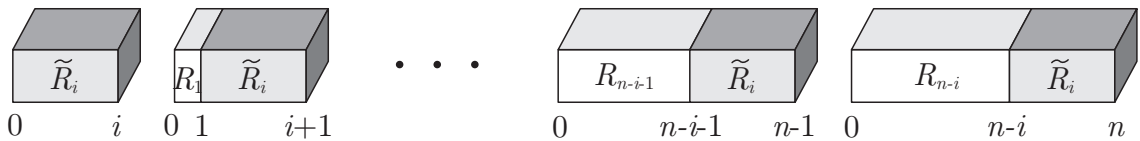}
	\caption{Tilings in case of Identity \ref{id:Rsum3d}}
	\label{fig:identity_Rsum3d}
\end{figure}

\end{document}